\documentclass[reqno]{amsart}
\usepackage{amsmath, amsthm, amssymb, amstext}

\usepackage[left=3.3cm,right=3.3cm,top=2cm,bottom=2cm,includeheadfoot]{geometry}
\usepackage{hyperref,xcolor}
\definecolor{bluecite}{HTML}{0875b7}
\hypersetup{
	pdfborder={0 0 0},
	colorlinks,
	linkcolor=bluecite,
	citecolor=bluecite
}
\usepackage{enumitem}
\newtheorem{theorem}{Theorem}
\newtheorem{remark}[theorem]{Remark}
\newtheorem{lemma}[theorem]{Lemma}
\newtheorem{proposition}[theorem]{Proposition}

\DeclareMathOperator*{\divergenz}{div}              %
\newcommand{\Ss}{\textnormal{(S}_+\textnormal{)}}

\newcommand*\diff{\mathop{}\!\mathrm{d}}

\newcommand{\N}{\mathbb{N}}

\newcommand{\R}{\mathbb{R}}

\newcommand{\W}{W^{1,\mathcal{H}}_{0}(\Omega)}
\newcommand{\Lp}[1]{L^{#1}(\Omega)}

\newcommand{\Wp}[1]{W^{1,#1}(\Omega)}

\newcommand{\Wpzero}[1]{W^{1,#1}_0(\Omega)}

\newcommand{\eps}{\varepsilon}
\newcommand{\ph}{\varphi}

\newcommand{\into}{\int_{\Omega}}

\newcommand{\weak}{\rightharpoonup}

\newcommand{\close}{\overline{\Omega}}

\renewcommand{\l}{\left}
\renewcommand{\r}{\right}
\numberwithin{theorem}{section}
\numberwithin{equation}{section}

\title[An existence result for singular Finsler double phase problems]{An existence result for singular Finsler double phase problems}

\author[C.\,Farkas]{Csaba Farkas}
\address[C.\,Farkas]{Sapientia Hungarian University of Transylvania, Department of Mathematics and Computer Science, T\^{a}rgu Mure\textcommabelow{s}, Romania}
\email{farkascs@ms.sapientia.ro, farkas.csaba2008@gmail.com}

\author[P.\,Winkert]{Patrick Winkert}
\address[P.\,Winkert]{Technische Universit\"{a}t Berlin, Institut f\"{u}r Mathematik, Stra\ss e des 17.\,Juni 136, 10623 Berlin, Germany}
\email{winkert@math.tu-berlin.de}


\begin{document}
	
\begin{abstract}
	In this paper, we study a class of singular double phase problems defined on Minkowski spaces in terms of Finsler manifolds and with right-hand sides that allow a certain type of critical growth for such problems. Under very general assumptions and based on variational tools we prove the existence of at least one nontrivial weak solution for such a problem. This is the first work dealing with a Finsler double phase operator even in the nonsingular case.
\end{abstract}
\maketitle

\section{Introduction}

Let $\Omega \subset \R^N$, $N\geq 2$, be a bounded domain with Lipschitz boundary $\partial \Omega$ and let $F\colon\mathbb R^N\to [0,\infty)$ be a convex function of class $C^2(\mathbb R^N\setminus \{0\})$ which is  absolutely homogeneous of degree one, that is,
\begin{equation*}
F(tx)=|t|F(x)\quad\text{for all } t\in \R \text{ and for all }x\in \R^N.
\end{equation*}
Then, one can consider the functional $E_F\colon H^1_0(\Omega)\to \R$ defined by
\begin{align*}
E_F(u)= \int_\Omega F^2(\nabla u)\diff x \quad \text{for }u\in H^1_0(\Omega),
\end{align*}
which is the energy functional  associated to the so-called Finsler Laplace operator given by
\begin{align*}
\Delta_F u= {\rm div} (F(\nabla u)\nabla F(\nabla u)).
\end{align*}

This class of problems arises in mathematical physics where the
objective is the minimization of $E_F\colon H^1_0(\Omega)\to \R$ under certain constraints (on perimeter or volume) while the solution corresponds to an optimal shape/configuration of anisotropic tension-surface. For instance,  the minimization of the functional $E_F\colon H^1_0(\Omega)\to \R$ explains  the specific
polyhedral shape of crystal structures in solid crystals with
sufficiently small grains, see Dinghas \cite{Dinghas-1944} and Taylor \cite{Taylor-2002}.

In the past years, many authors studied elliptic equations involving the  Finsler Laplace operator.  We refer, for instance, to the works of Cianchi-Salani \cite{Cianchi-Salani}  and Wang-Xia \cite{Wang-Xia-2011} who treated the Serrin-type overdetermined anisotropic problem involving the function $F$ and the operator $\Delta_F$ while  Farkas-Fodor-Krist\'aly \cite{Farkas-Fodor-Kristaly-2015} considered a sublinear elliptic problem with Dirichlet boundary condition.  Further studies of anisotropic phenomena can be found in the papers of Bellettini-Paolini \cite{Bellettini-Paolini-1996}, Belloni-Ferone-Kawohl \cite{Belloni-Ferone-Kawohl-2003}, Della Pietra-Gavitone \cite{DellaPietra-Gavitone-2016}, Della Pietra-di Blasio-Gavitone \cite{DellaPietra-diBlasio-Gavitone-2020}, Della Pietra-Gavitone-Piscitelli \cite{DellaPietra-Gavitone-Piscitelli-2019} and   Ferone-Kawohl \cite{Ferone-Kawohl-2009}; see also the references therein. Nonlinear equations on general Finsler manifolds can be found in Farkas-Krist\'aly-Varga \cite{Farkas-Kristaly-Varga-2015}, Farkas \cite{Farkas-2020} and the references therein.

In this paper we are concerned with the following singular Finsler double phase problem given in the form
\begin{equation}\label{problem}
\begin{aligned}
-\divergenz\left(F^{p-1}(\nabla u)\nabla F(\nabla u) +\mu(x) F^{q-1}(\nabla u)\nabla F(\nabla u)\right) & = u^{p^{*}-1}+\lambda\l(u^{\gamma-1}+ g(u)\r) && \text{in } \Omega,\\
u & > 0 &&\text{in } \Omega,\\
u & = 0 &&\text{on } \partial \Omega,
\end{aligned}
\end{equation}
where $\lambda>0$ is a positive parameter and  the following assumptions are supposed:
\begin{enumerate}
	\item[\textnormal{(H)}]
	\begin{enumerate}
		\item[\textnormal{(i)}]
		$0<\gamma<1$ and
		\begin{align}\label{condition_poincare}
		2 \leq p<q<N, \qquad  \frac{q}{p}<1+\frac{1}{N}, \qquad 0 \leq \mu(\cdot)\in C^{0,1}(\close);
		\end{align}
		\item[\textnormal{(ii)}] 
		the function $g\colon\R\to\R$ is continuous and there exist $1<\theta<p\leq\nu<p^*$ as well as nonnegative constants $c_{i}$, $i=1,2$ such that
		\begin{align*}
		g(s)\leq c_{1}s^{\nu-1}+c_{2}s^{\theta-1}\quad \text{for all } s\geq 0,
		\end{align*}
		where $p^*$ is the critical exponent to $p$ given by
		\begin{align}\label{critical_exponent}
		p^*:=\frac{Np}{N-p};
		\end{align}
		\item[\textnormal{(iii)}]
		the function $F\colon\R^N \to [0,\infty)$ is a positively homogeneous
		Minkowski norm.
	\end{enumerate}
\end{enumerate}

The differential operator in \eqref{problem} is called the Finsler double phase operator. It is defined by
\begin{align}\label{operator_double_phase}
\divergenz\left(F^{p-1}(\nabla u)\nabla F(\nabla u) +\mu(x) F^{q-1}(\nabla u)\nabla F(\nabla u)\right) \quad \text{for }u\in \W,
\end{align}
where $\W$ is a suitable Musielak-Orlicz Sobolev space which is defined in Section \ref{section_2}.  If $\mu\equiv 0$, the operator in \eqref{operator_double_phase} becomes the Finsler $p$-Laplacian (which reduces to the Finsler Laplacian in case $p=2$) and if $\inf_{\close} \mu>0$, we have the weighted Finsler $(q,p)$-Laplacian.  Clearly, when $\mu\equiv0$, \eqref{problem} is strongly related to following Brezis-Nirenberg type problem
\begin{equation*}
\begin{aligned}
-\Delta_{p} u & =|u|^{p^*-2}u+\lambda g(u)\quad && \text{in } \Omega,\\
u & = 0 &&\text{on } \partial \Omega.
\end{aligned}
\end{equation*}

Note that the Finsler $p$-Laplacian is often called anisotropic $p$-Laplacian. Since the word ``anisotropic'' also stands in the literature for problems where the exponent $p$ is a function, we avoid this denotation and use only ``Finsler''.

The operator \eqref{operator_double_phase} is related to the energy functional
\begin{align}\label{integral_minimizer}
\omega \mapsto \int_\Omega \big(F^p(\nabla  \omega)+\mu(x)F^q(\nabla  \omega)\big)\diff x,
\end{align}
where the integrand $H(x,\xi)=F^p(\xi)+\mu(x)F^q(\xi)$ for all $(x,\xi) \in \Omega\times \R^N$ has unbalanced growth, that is,
\begin{align*}
b_1 |\xi|^p \leq F^p(\xi)\leq H(x,\xi) \leq b_2 \l(1+F^q(\xi)\r)\leq  b_3 \l(1+|\xi|^q\r)
\end{align*}
for a.\,a.\,$x\in\Omega$ and for all $\xi\in\R^N$ with $b_1,b_2,b_3>0$. The integral functional \eqref{integral_minimizer} is characterized by the fact that the energy density changes its ellipticity and growth properties according to the point in the domain. To be more precise, the growth rate of \eqref{integral_minimizer} on the set $\{x\in \Omega: \mu(x)=0\}$ is of order $p$ and on the set $\{x\in \Omega: \mu(x) \neq 0\}$ it is of order $q$. So the integrand $H$ switches between two different phases of elliptic behaviours.

When $F$ coincide with the Euclidean norm, Zhikov \cite{Zhikov-1986} was the first who studied functionals of type \eqref{integral_minimizer}  for models which describe strongly anisotropic materials, see also the monograph of Zhikov-Kozlov-Oleinik \cite{Zhikov-Kozlov-Oleinik-1994}. In the last decades, several authors investigated functionals of the form in \eqref{integral_minimizer} concerning regularity of (local) minimizers. We mention the papers of Baroni-Colombo-Mingione \cite{Baroni-Colombo-Mingione-2015,Baroni-Colombo-Mingione-2016,Baroni-Colombo-Mingione-2018}, Baroni-Kuusi-Mingione \cite{Baroni-Kuusi-Mingione-2015}, Co\-lom\-bo-Mingione \cite{Colombo-Mingione-2015a,Colombo-Mingione-2015b},
Marcellini \cite{Marcellini-1989,Marcellini-1991}, Ok \cite{Ok-2018,Ok-2020}, Ragusa-Tachikawa \cite{Ragusa-Tachikawa-2020} and the references therein.

Before we state the main result of the present paper, we recall the notion of the weak solution. A function $u\in \W$ is called a weak solution of problem \eqref{problem} if $u^{\gamma-1}\varphi\in L^{1}(\Omega)$, $u>0$ for a.\,a.\,$x\in\Omega$ and  if
\begin{align*}
\begin{split}
& \int_{\Omega}\left(F^{p-1}(\nabla u)\nabla F(\nabla u) +\mu(x) F^{q-1}(\nabla u)\nabla F(\nabla u)\right)\cdot \nabla\varphi\diff x\\ &=\int_{\Omega}u^{p^{*}-1}\varphi\diff x+\lambda\int_{\Omega}u^{\gamma-1}\varphi\diff x+\lambda\int_{\Omega}g(u)\varphi\diff x
\end{split}
\end{align*}
is satisfied for all $\varphi\in\W$.

Now, we are in the position to formulate our main result.
\begin{theorem}\label{main_theorem}
	Let hypotheses \textnormal{(H)} be satisfied. Then there exists $\lambda_*>0$ such that for every $\lambda\in (0,\lambda_*)$, problem \eqref{problem} has a nontrivial weak solution. 
\end{theorem}

\begin{remark}
	Since we are looking for positive solutions and hypothesis \textnormal{(H)(ii)} concerns the positive semiaxis $\R_+=[0,+\infty)$, without any loss of generality, we may assume that $g(s)=0$ for all $s\leq 0$.
\end{remark}

To the best of our knowledge, this is the first paper which considers a double phase operator on a Minkowski space in the sense of Finsler geometry, see Bao-Chern-Shen \cite{Bao-Chern-Shen-2000}. Even if $F$ is the Euclidean norm, there is only one work for singular double phase problems done by Chen-Ge-Wen-Cao \cite{Chen-Ge-Wen-Cao-2020} who proved the existence of a nontrivial weak solution with negative energy. Recent existence results for double phase problems in the case if $F$ is the Euclidean norm and without singular term can be found, for example, in the papers of Colasuonno-Squassina \cite{Colasuonno-Squassina-2016}, Liu-Dai \cite{Liu-Dai-2018},  Gasi\'nski-Papageorgiou \cite{Gasinski-Papageorgiou-2019}, Gasi\'nski-Winkert \cite{Gasinski-Winkert-2020a, Gasinski-Winkert-2020b, Gasinski-Winkert-2021}, Perera-Squassina \cite{Perera-Squassina-2019} and the references therein. 

So far, it is not clear what is the optimal critical exponent for Musielak-Orlicz Sobolev spaces even when $F$ is the Euclidean norm, see, for example, the paper of Fan \cite{Fan-2012}. For the moment $p^*$ seems to be best exponent (it is probably not optimal) and only a continuous (in general noncompact) embedding from $\W \hookrightarrow \Lp{p^*}$ is available. So we call it ``type of critical growth''. 

Finally, we mention recent papers which are very close to our topic. We refer to the works of Bahrouni-R\u{a}dulescu-Winkert \cite{Bahrouni-Radulescu-Winkert-2020}, Marino-Winkert \cite{Marino-Winkert-2020}, Papageorgiou-R\u{a}dulescu-Repov\v{s} \cite{Papageorgiou-Radulescu-Repovs-2020}, Papageorgiou-Winkert \cite{Papageorgiou-Winkert-2021,Papageorgiou-Winkert-2019},  Zeng-Bai-Gasi\'nski-Winkert  \cite{Zeng-Gasinski-Winkert-Bai-2020, Zeng-Bai-Gasinski-Winkert-2020} and the references therein.

The paper is organized as follows. In Section \ref{section_2} we recall some basic properties of the Minkowski norm and introduce appropriate Musielak-Orlicz Sobolev spaces equipped with the Minkowski norm $F$. In Section \ref{section_3} we first present the Minkowskian version of the inequality
\begin{align*}
|b|^r \geq |a|^r+r |a|^{r-2}a \cdot (b-a) + 2^{1-r} |a-b|^r\quad \text{for }r\geq 2 \text{ and for all }a,b \in \R^N,
\end{align*}
see Lindqvist \cite{Lindqvist-1990}, which is the basis for the proof of the sequentially weakly lower semicontinuity of the corresponding energy functional to problem \eqref{problem} on closed convex subsets of $\W$ not necessarily being balls. The  sequentially weakly lower semicontinuity is one of the main steps in order to prove our main result. At the end, based on the direct method of calculus along with variational tools, we are going to prove Theorem \ref{main_theorem}. Our approach follows ideas of the paper of Faraci-Farkas \cite{Faraci-Farkas-2015} in which the authors consider a singular problem driven by the $p$-Laplacian, that is, $\mu\equiv0$ and $F(\xi)=\left(\sum_{i=1}^N |\xi_i|^2\right)^{1/2}$. 

\section{Preliminaries}\label{section_2}

In this section we will recall the main properties of the Minkowski space $(\R^N,F)$ and the functional setting for the double phase operator.

Let $F\colon \R^N\to [0,\infty)$ be a positively homogeneous Minkowski norm, that is, $F$ is a positive homogeneous function and verifies the properties
\begin{enumerate}
	\item[$\bullet$] 
	$F\in C^\infty(\mathbb R^N\setminus \{0\})$;
	\item[$\bullet$]  
	The Hessian matrix $\nabla^2 ({F^2}/{2})(x)$ is positive definite for all $x\neq 0$.
\end{enumerate} 

Note that the pair $(\mathbb R^N,F)$ is the simplest not necessarily reversible Finsler manifold whose flag curvature is identically zero, the geodesics are straight lines and the intrinsic distance between two points $x,y\in \mathbb R^N$ is given by
\begin{equation*}
d_F(x,y)=F(y-x).
\end{equation*}
In fact, $(\R^N,d_F)$ is a quasi-metric space and in general, we have $d_F(x,y)\neq d_F(y,x)$. 

A typical example for a Minkowski norm is the so called  Randers metric which is defined by 
\begin{align*}
F(x)=\sqrt{\langle Ax ,x \rangle}+\langle b,x \rangle,
\end{align*}
where $A$ is a  positive definite and symmetric  $(N\times N)$-type matrix and  $b=(b_i)\in \mathbb R^N$ is a fixed vector such that $\sqrt{\langle A^{-1}b,b\rangle}<1$. The pair $(\mathbb R^N,F)$ is often called Randers space which describes the electromagnetic field of the physical space-time in general relativity, see Randers \cite{Randers-1941}. They are deduced as the solution of the Zermelo navigation problem.

In what follows we recall some basic properties of $F$, see Bao-Chern-Shen \cite[\S 1.2]{Bao-Chern-Shen-2000}.

\begin{proposition}\label{basic-properties}
	Let $F\colon\R^N\to [0,\infty)$ be a positively homogeneous Minkowski norm. Then, the following assertions hold true:
	\begin{enumerate}
		\item[\textnormal{(i)}] 
		Positivity: $F(x)>0$ for all $x\neq 0$;
		\item[\textnormal{(ii)}]
		Convexity: $F$ and $F^2$ are strictly convex;
		\item[\textnormal{(iii)}]
		Euler's theorem: $x\cdot \nabla F(x)=F(x)$ and 
		\begin{align*}
		\nabla^2 ({F^2}/{2})({x}){x}\cdot {x}=F^2(x)\quad  \text{for all } x\in\R^N\setminus \{0\};
		\end{align*}
		\item[\textnormal{(iv)}]
		Homogeneity: $\nabla F(tx)=\nabla F(x)$ and 
		\begin{align*}
		\nabla^2 F^2(tx)=\nabla^2 F^2(x)\quad \text{for all }x\in \R^N\setminus \{0\} \text{ and for all } t> 0.
		\end{align*}
	\end{enumerate}
\end{proposition}

For $1\leq r<\infty$, we denote by $\Lp{r}$ and $L^r(\Omega;\R^N)$ the usual Lebesgue spaces equipped with the norm $\|\cdot\|_r$.  If $1<r<\infty$, then $\Wp{r}$ and $\Wpzero{r}$ stand for the Sobolev spaces endowed with the norms

\begin{align*}
\|u\|_{1,r,F}= \|F(\nabla u) \|_{r}+\|u\|_{r},
\end{align*}
and
\begin{align*}
\|u\|_{1,r,0,F}= \|F(\nabla u) \|_{r},
\end{align*}
respectively. Now, let $\mathcal{H}\colon \Omega \times [0,\infty)\to [0,\infty)$ be the function defined by
\begin{align*}
(x,t)\mapsto t^p+\mu(x)t^q,
\end{align*}
where \eqref{condition_poincare} is satisfied. The Musielak-Orlicz space $L^\mathcal{H}(\Omega)$ is defined by
\begin{align*}
L^\mathcal{H}(\Omega)=\left \{u ~ \Big | ~ u: \Omega \to \R \text{ is measurable and } \into \mathcal{H}(x,|u|)\diff x < +\infty \right \}
\end{align*}
equipped with the Luxemburg norm
\begin{align*}
\|u\|_{\mathcal{H}} = \inf \left \{ \tau >0\,:\, \rho_{\mathcal{H}}\left(\frac{u}{\tau}\right) \leq 1  \right \},
\end{align*}
where
\begin{align}\label{modular}
\rho_{\mathcal{H}}(u):=\into \mathcal{H}(x,|u|)\,dx=\into \big(|u|^{p}+\mu(x)|u|^q\big)\diff x.
\end{align}
We know that $L^\mathcal{H}(\Omega)$ is a reflexive Banach space.

The norm $\|\cdot\|_{\mathcal{H}}$ and the modular function $\rho_\mathcal{H}$ are related as follows, see Liu-Dai \cite[Proposition 2.1]{Liu-Dai-2018}.

\begin{proposition}\label{proposition_modular_properties}
	Let \eqref{condition_poincare} be satisfied, let $y\in \Lp{\mathcal{H}}$ and let $\rho_{\mathcal{H}}$ be defined by \eqref{modular}.
	\begin{enumerate}
		\item[\textnormal{(i)}]
		If $y\neq 0$, then $\|y\|_{\mathcal{H}}=\lambda$ if and only if $ \rho_{\mathcal{H}}(\frac{y}{\lambda})=1$;
		\item[\textnormal{(ii)}]
		$\|y\|_{\mathcal{H}}<1$ (resp.\,$>1$, $=1$) if and only if $ \rho_{\mathcal{H}}(y)<1$ (resp.\,$>1$, $=1$);
		\item[\textnormal{(iii)}]
		If $\|y\|_{\mathcal{H}}<1$, then $\|y\|_{\mathcal{H}}^q\leq \rho_{\mathcal{H}}(y)\leq\|y\|_{\mathcal{H}}^p$;
		\item[\textnormal{(iv)}]
		If $\|y\|_{\mathcal{H}}>1$, then $\|y\|_{\mathcal{H}}^p\leq \rho_{\mathcal{H}}(y)\leq\|y\|_{\mathcal{H}}^q$;
		\item[\textnormal{(v)}]
		$\|y\|_{\mathcal{H}}\to 0$ if and only if $ \rho_{\mathcal{H}}(y)\to 0$;
		\item[\textnormal{(vi)}]
		$\|y\|_{\mathcal{H}}\to +\infty$ if and only if $ \rho_{\mathcal{H}}(y)\to +\infty$.
	\end{enumerate}
\end{proposition}

The corresponding Sobolev space $W^{1,\mathcal{H}}(\Omega)$ is defined by
\begin{align*}
W^{1,\mathcal{H}}(\Omega)= \left \{u \in L^\mathcal{H}(\Omega) \,:\, F(\nabla u) \in L^{\mathcal{H}}(\Omega) \right\}
\end{align*}
equipped with the norm
\begin{align*}
\|u\|_{1,\mathcal{H},F}= \|F(\nabla u) \|_{\mathcal{H}}+\|u\|_{\mathcal{H}}.
\end{align*}
By $W^{1,\mathcal{H}}_0(\Omega)$ we denote the completion of $C^\infty_0(\Omega)$ in $W^{1,\mathcal{H}}(\Omega)$ and thanks to \eqref{condition_poincare} we have an equivalent norm on $W^{1,\mathcal{H}}_0(\Omega)$ given by
\begin{align*}
\|u\|_{1,\mathcal{H},0,F}=\|F(\nabla u)\|_{\mathcal{H}},
\end{align*}
see Colasuonno-Squassina \cite[Proposition 2.18]{Colasuonno-Squassina-2016}.
Both spaces $W^{1,\mathcal{H}}(\Omega)$ and $W^{1,\mathcal{H}}_0(\Omega)$ are uniformly convex and so, reflexive Banach spaces. 

Let $\rho_{\mathcal{H},F}\colon \W\to\R$ be the corresponding modular to $\W$ defined by
\begin{align}\label{modular2}
\rho_{\mathcal{H},F}(u):=\into \big(F^p(\nabla u)+\mu(x)F^q(\nabla u)\big)\diff x.
\end{align}
Similar to Proposition \ref{proposition_modular_properties}, we have the same relations between the norm $\|\cdot\|_{1,\mathcal{H},0,F}$ and the modular $\rho_{\mathcal{H},F}$.
\begin{proposition}
	Let \eqref{condition_poincare} and \textnormal{(H)(iii)} be satisfied and let $\rho_{\mathcal{H},F}$ be defined by \eqref{modular2}. For $y\in \W$ we have the following assertions.
	\begin{enumerate}
		\item[\textnormal{(i)}]
		If $y\neq 0$, then $\|y\|_{1,\mathcal{H},0,F}=\lambda$ if and only if $ \rho_{\mathcal{H},F}(\frac{y}{\lambda})=1$;
		\item[\textnormal{(ii)}]
		$\|y\|_{1,\mathcal{H},0,F}<1$ (resp.\,$>1$, $=1$) if and only if $ \rho_{\mathcal{H},F}(y)<1$ (resp.\,$>1$, $=1$);
		\item[\textnormal{(iii)}]
		If $\|y\|_{1,\mathcal{H},0,F}<1$, then $\|y\|_{1,\mathcal{H},0,F}^q\leq \rho_{\mathcal{H},F}(y)\leq\|y\|_{1,\mathcal{H},0,F}^p$;
		\item[\textnormal{(iv)}]
		If $\|y\|_{1,\mathcal{H},0,F}>1$, then $\|y\|_{1,\mathcal{H},0,F}^p\leq \rho_{\mathcal{H},F}(y)\leq\|y\|_{1,\mathcal{H},0,F}^q$;
		\item[\textnormal{(v)}]
		$\|y\|_{1,\mathcal{H},0,F}\to 0$ if and only if $ \rho_{\mathcal{H},F}(y)\to 0$;
		\item[\textnormal{(vi)}]
		$\|y\|_{1,\mathcal{H},0,F}\to +\infty$ if and only if $ \rho_{\mathcal{H},F}(y)\to +\infty$.
	\end{enumerate}
\end{proposition}

In addition it is known that the embedding
\begin{align*}
\W \hookrightarrow \Lp{r}
\end{align*}
is continuous whenever $r\leq p^*$ and compact whenever $r<p^*$, see Colasuonno-Squassina \cite[Proposition 2.15]{Colasuonno-Squassina-2016}. Recall that $p^*$ is the critical exponent to $p$, see \eqref{critical_exponent}.

Let $A\colon \W\to \W^*$ be the Finsler double phase operator defined by
\begin{align}\label{operator_representation}
\langle A(u),\ph\rangle_{\mathcal{H},F} :=\into \left(F^{p-1}(\nabla u)\nabla F(\nabla u) +\mu(x) F^{q-1}(\nabla u)\nabla F(\nabla u) \right)\cdot\nabla\ph \diff x,
\end{align}
where $\langle \cdot,\cdot\rangle_{\mathcal{H},F}$ is the duality pairing between $\W$ and its dual space $\W^*$. The properties of the operator $A\colon \W\to \W^*$ are summarized in the following proposition, see Liu-Dai \cite{Liu-Dai-2018}, by taking the properties of $F$ into account.

\begin{proposition}
	The operator $A$ defined by \eqref{operator_representation} is bounded, continuous, monotone (hence maximal monotone) and of type $\Ss$.
\end{proposition}

\section{Proof of the main result}\label{section_3}

In this section we present the proof of Theorem \ref{main_theorem}. We start with a lemma which is the Minkowskian version of the Lindqvist inequality; see Lindqvist \cite{Lindqvist-1990}. Note that this inequality was proved for general Finsler manifolds by the first author \cite{Farkas-2020}. For the sake of completeness we will also provide the proof here.

\begin{lemma}
	Let $F\colon\R^N\to [0,\infty)$ be a positively homogeneous Minkowski norm. For every $\xi,\beta\in\mathbb{\mathbb{R}}^{N}$ and for every $r\geq2$ we have
	\begin{equation}\label{eq:important}
	F^{r}(\beta)\geq F^{r}(\xi)+rF^{r-1}(\xi)\nabla F(\xi)\cdot \l(\beta-\xi\r)+\frac{l_{F}^{\frac{r}{2}}}{2^{r-1}}F^{r}(\beta-\xi),
	\end{equation}
	where 
	\begin{align*}
	l_{F}=\min\left\{  \nabla^2\left(\frac{F^{2}}{2}\right)(x)y\cdot y \,:\, F(x)=F(y)=1\right\}.
	\end{align*}
	
\end{lemma}

\begin{proof}
	From the definition of $l_F$, one has that 
	\begin{equation}\label{lFimportant}
	F^{2}({t\xi+(1-t)\beta} ) \le tF^{2}(\xi) +(1-t)F^{2}(\beta)
	-{l_{F}}t(1-t) F^{2}(\beta-\xi) 
	\end{equation}
	for all $\xi,\beta\in \R^N$ and for all $t \in [0,1]$, see Farkas-Krist\'{a}ly-Varga \cite[(2.19) on p. 1228]{Farkas-Kristaly-Varga-2015} or Ohta \cite{Ohta-2009}.
	Choosing $t=\frac{1}{2}$ in \eqref{lFimportant}  gives 
	\begin{equation}\label{eq:l_F_2}
	F^{2}\left(\frac{\xi+\beta}{2}\right)\le\frac{F^{2}(\xi)+F^{2}(\beta)}{2}-\frac{l_{F}}{4}F^{2}(\beta-\xi).
	\end{equation}
	Note that for nonnegative $a,b \in \R$ we have
	\begin{align*}
	\l(a^r+b^r\r)^{\frac{1}{r}}=\l(\l(a^r+b^r\r)^{\frac{2}{r}}\r)^{\frac{1}{2}}\leq \l(2^{\frac{2}{r}-1}\l(a^{r\frac{2}{r}}+b^{r\frac{2}{r}}\r)\r)^{\frac{1}{2}}\leq\l(a^2+b^2\r)^{\frac{1}{2}}.
	\end{align*}
	Applying the inequality above in combination with \eqref{eq:l_F_2} leads to
	\begin{align}\label{eq:lem1-1}
	\begin{split}
	F^{r}\left(\frac{\xi+\beta}{2}\right)+\frac{l_{F}^{\frac{r}{2}}}{2^{r}}F^{^{r}}(\beta-\xi) 
	& \leq \left(F^{2}\left(\frac{\xi+\beta}{2}\right)+\frac{l_{F}}{4}F^{2}(\beta-\xi)\right)^{\frac{r}{2}}\leq\left(\frac{F^{2}(\xi)+F^{2}(\beta)}{2}\right)^{\frac{r}{2}}.
	\end{split}
	\end{align}
	By the convexity of the function $(\cdot)^{\frac{r}{2}}$ we have
	\begin{align}\label{eq:lem1-2}
	\left(\frac{F^{2}(\xi)+F^{2}(\beta)}{2}\right)^{\frac{r}{2}}\leq\frac{1}{2}F^{r}(\xi)+\frac{1}{2}F^{r}(\beta).
	\end{align}
	On the other hand, by convexity, one has
	\begin{align}\label{eq:lem1-3}
	F^{{r}}\left(\frac{\xi+\beta}{2}\right)\geq F^{r}(\xi)+\frac{r}{2}F^{r-1}(\xi)\nabla F(\xi)\cdot \l(\beta-\xi\r).
	\end{align}
	Combining \eqref{eq:lem1-1}, \eqref{eq:lem1-2} and \eqref{eq:lem1-3} yields
	\begin{align*}
	\frac{r}{2}F^{r-1}(\xi)\nabla F(\xi)\cdot \l(\beta-\xi\r)+\frac{l_{F}^{\frac{r}{2}}}{2^{r}}F^{{r}}(\beta-\xi)+\frac{1}{2}F^{{r}}(\xi)\leq\frac{1}{2}F^{{r}}(\beta),
	\end{align*}
	for all $\xi,\beta\in \R^N$. Multiplying the last inequality by the factor $2$ proves the assertion of the lemma.
\end{proof}

Note that the constant $l_F$ is the uniformity constant of $F$ and $l_F\in [0,1]$, see Farkas-Krist\'aly-Varga \cite{Farkas-Kristaly-Varga-2015} and Otha \cite{Ohta-2009}.

Let $J_\lambda\colon \W\to \R $ be the energy functional corresponding to problem \eqref{problem} given by
\begin{align*}
J_\lambda(u)=\int_{\Omega}\left(\frac{1}{p}F^{p}(\nabla u)+\frac{\mu(x)}{q}F^{q}(\nabla u)\right)\diff x-\frac{1}{p^{*}}\int_{\Omega}(u_{+})^{p^{*}}\diff x-\frac{\lambda}{\gamma}\int_{\Omega}(u_{+})^{\gamma}\diff x-\lambda \int_{\Omega}G(u_{+})\diff x,
\end{align*}
where $u_{\pm}=\max\{\pm u,0\}$ and $G(s)=\int^s_0 g(t)\diff t$. It is clear that $J_\lambda$ does not belong to $C^1(\W)$ due to the appearance of the singular term. 

Furthermore we denote by $\kappa_{p^{*}}$ the inverse of the Sobolev embedding constant of $W_{0}^{1,p}(\Omega)\hookrightarrow L^{p^{*}}(\Omega)$, that is,
\begin{align*}
(\kappa_{p^{*}})^{-1}=\inf_{\underset{u\neq0}{u\in W_{0}^{1,p}(\Omega),}}\frac{\|u\|_{1,p,0,F}}{\|u\|_{p^{*}}}.
\end{align*}

Now, we are going to prove that the energy functional $J_\lambda\colon \W\to \R $ is sequentially weakly lower semicontinuous on closed convex subsets of $\W$.

\begin{proposition}\label{proposition_3.1}
	For every $\lambda>0$ and for every $\sigma \in (0,\sigma^*)$, where
	\begin{align*}
	\sigma^{*}:={\displaystyle \left(\frac{p^{*}}{p}\frac{l_{F}^{\frac{p}{2}}}{2^{p-1}\kappa_{p^{*}}^{p^{*}}}\right)^{\frac{1}{p^{*}-p}}},
	\end{align*}
	the restriction  of $J_\lambda$ to the closed convex set $B_{\sigma}$ which is given by
	\begin{align*}
	B_{\sigma}:=\l\{u\in \W \, : \, \|u\|_{1,p,0,F}\leq\sigma\r\},
	\end{align*}
	is sequentially weakly lower semicontinuous.
\end{proposition}

\begin{proof}
	Let $\lambda>0$ and let $\sigma \in (0,\sigma^*)$. We are going to prove that the functional $J_\lambda\colon \W\to \R $ is weakly lower semicontinuous on $B_{\sigma}$. To this end, let $\{u_{n}\}_{n\in\N}\subset B_{\sigma}$ be such that $u_{n}\weak u$ in $\W$ and let 
	\begin{align*}
	I(u)=\int_{\Omega}\left(\frac{1}{p}F^{p}(\nabla u)+\frac{\mu(x)}{q}F^{q}(\nabla u)\right)\diff x-\frac{1}{p^{*}}\int_{\Omega}(u_{+})^{p^{*}}\diff x.
	\end{align*}
	Note that, up to a subsequence if necessary, we have for $r<p^{*}$, $u_{n}\to u$ in $L^{r}(\Omega)$ and $\nabla u_{n}\weak \nabla u$ in $L^{p}(\Omega)$. Therefore, it suffices to prove that
	\begin{align*}
	\liminf_{n\to\infty}\left(I(u_n)-I(u)\right)\geq 0,
	\end{align*}
	see hypothesis \textnormal{(H)(ii)}. 
	
	Applying \eqref{eq:important} we get the following inequalities
	\begin{align*}
	& \int_{\Omega}\frac{1}{p}F^{p}(\nabla u_{n})\diff x\\ &\geq\int_{\Omega}\frac{1}{p}F^{p}(\nabla u)\diff x+\int_{\Omega}F^{p-1}(\nabla u)\nabla F(\nabla u)\cdot \nabla \l(u_{n}-u\r)\diff x+\frac{l_{F}^{\frac{p}{2}}}{p2^{p-1}}\int_{\Omega}F^{p}(\nabla u_{n}-\nabla u)\diff x,
	\end{align*}
	and 
	\begin{align*}
	\int_{\Omega}\frac{\mu(x)}{q}F^{q}(\nabla u_{n})\diff x 
	& \geq\int_{\Omega}\frac{\mu(x)}{q}F^{q}(\nabla u)\diff x+\int_{\Omega}\mu(x)F^{q-1}(\nabla u)\nabla F(\nabla u)\cdot \nabla \l(u_{n}- u\r)\diff x\\
	& \quad +\frac{l_{F}^{\frac{q}{2}}}{q2^{q-1}}\int_{\Omega}\mu(x)F^{q}(\nabla u_{n}-\nabla u)\diff x.
	\end{align*}
	On the other hand, by the Brezis-Lieb lemma, see, for example, Papageorgiou-Winkert \cite[Lemma 4.1.22]{Papageorgiou-Winkert-2018}, we have
	\begin{align*}
	\liminf_{n\to\infty}\left(\int_{\Omega}|u_{n}|^{p^{\ast}}\diff x-\int_{\Omega}|u|^{p^{\ast}}\diff x\right)=\liminf_{n\to\infty}\int_{\Omega}|u_{n}-u|^{p^{\ast}}\diff x.
	\end{align*}
	Therefore since 
	\begin{align*}
	& {\int_{\Omega}F^{p-1}(\nabla u)\nabla F(\nabla u)\cdot \nabla \l(u_{n}- u\r)\diff x\to 0},\\
	&{\displaystyle \int_{\Omega}\mu(x)F^{q-1}(\nabla u)\nabla F(\nabla u)\cdot \nabla \l(u_{n}-u\r)\diff x\to 0},\\
	&{\frac{l_{F}^{\frac{q}{2}}}{q2^{q-1}}\int_{\Omega}\mu(x)F^{q}(\nabla u_{n}-\nabla u)\diff x\geq0},
	\end{align*}
	one has
	\begin{align*}
	\liminf_{n\to\infty}\left(I(u_{n})-I(u)\right) 
	& \geq\liminf_{n\to\infty}\left(\frac{l_{F}^{\frac{p}{2}}}{p2^{p-1}}\int_{\Omega}F^{p}(\nabla u_{n}-\nabla u)\diff x-\frac{1}{p^{*}}\int_{\Omega}|u_{n}-u|^{p^{\ast}}\diff x\right)\\
	& \geq\liminf_{n\to\infty}\left(\frac{l_{F}^{\frac{p}{2}}}{p2^{p-1}}\|u_{n}-u\|_{1,p,0,F}^{p}-\frac{1}{p^{*}}\kappa_{p^{*}}^{p^{*}}\|u_{n}-u\|_{1,p,0,F}^{p^{*}}\right)\\
	& \geq \liminf_{n\to\infty}\|u_{n}-u\|_{1,p,0,F}^{p}\left(\frac{l_{F}^{\frac{p}{2}}}{p2^{p-1}}-\frac{1}{p^{*}}\kappa_{p^{*}}^{p^{*}}\sigma^{p^{*}-p}\right)\geq0,
	\end{align*}
	which proves the assertion of the proposition. 
\end{proof}

For further use, we introduce the functionals $I_1\colon\W\to\mathbb{R}$ and $I_2\colon\Lp{\mathcal{H}}\to\R$ defined by 
\begin{align*}
{I_1}(u)=-\frac{1}{q}\int _{\Omega}\mu(x)F^{q}(\nabla u)\diff x+\frac{1}{p^{*}}\int _{\Omega}(u_{+})^{p^{*}}+\frac{\lambda}{\gamma}\int _{\Omega}(u_{+})^{\gamma}\diff x+\lambda \int _{\Omega}G(u_{+})\diff x,
\end{align*}
and 
\begin{align*} 
I_2(u)=\frac{1}{p^{*}}\int _{\Omega}(u_{+})^{p^{*}}+\frac{\lambda}{\gamma}\int _{\Omega}(u_{+})^{\gamma}\diff x+\lambda\int _{\Omega}G(u_{+})\diff x.
\end{align*}

Now we are in the position to prove Theorem \ref{main_theorem}.

\begin{proof}[Proof of Theorem \ref{main_theorem}] 
	Let $\lambda>0$ and let $\sigma \in (0,\sigma^*)$, where $\sigma^*$ is as in Proposition \ref{proposition_3.1}. We define 
	\begin{align*}
	\varphi_{\lambda}(\sigma):=\inf_{\|u\|_{1,p,0,F}<\sigma}\frac{\sup_{B_{\sigma}}I_1-I_1(u)}{\sigma^{p}-\|u\|_{1,p,0,F}^{p}}
	\end{align*}
	and 
	\begin{align*}
	\psi_{\lambda}(\sigma):=\sup_{B_{\sigma}}I_1.
	\end{align*}
	
	We claim that there exist $\lambda,\sigma>0$ small enough such that 
	\begin{equation}
	\varphi_{\lambda}(\sigma)<\frac{1}{p}.\label{min}
	\end{equation}
	We claim that \eqref{min} holds if there exist $\lambda,\sigma>0$ such that 
	\begin{equation}\label{eq:whatweeneed}
	\inf_{\xi<\sigma}\frac{\psi_{\lambda}(\sigma)-\psi_{\lambda}(\xi)}{\sigma^{p}-\xi^{p}}<\frac{1}{p}.
	\end{equation}
	
	Note that, by taking $\xi=\sigma-\varepsilon$ for some $\eps \in (0,\sigma)$, it follows
	\begin{align*}
	\frac{\psi_{\lambda}(\sigma)-\psi_{\lambda}(\xi)}{\sigma^{p}-\xi^{p}} & 
	= \frac{\psi_{\lambda}(\sigma)-\psi_{\lambda}(\sigma-\varepsilon)}{\sigma^{p}-(\sigma-\varepsilon)^{p}}\\
	& = \frac{\psi_{\lambda}(\sigma)-\psi_{\lambda}(\sigma-\varepsilon)}{\varepsilon}\cdot\frac{-\frac{\varepsilon}{\sigma}}{\sigma^{p-1}[(1-\frac{\varepsilon}{\sigma})^{p}-1]}.
	\end{align*}
	So passing to the limit as $\varepsilon\to 0$, we get that \eqref{eq:whatweeneed} is fulfilled if 
	\begin{equation}
	\limsup_{\varepsilon\to0^{+}}\frac{\psi_{\lambda}(\sigma)-\psi_{\lambda}(\sigma-\varepsilon)}{\varepsilon}<\sigma^{p-1}.\label{eq:ezkell}
	\end{equation}
	
	First note that
	\begin{align*}
	\frac{1}{\varepsilon}\left|\psi_{\lambda}(\sigma)-\psi_{\lambda}(\sigma-\varepsilon)\right| & =\frac{1}{\varepsilon}\left|\sup_{v\in B_{1}}I_1(\sigma v)-\sup_{v\in B_{1}}I_1((\sigma-\varepsilon)v)\right|\\
	& \leq\frac{1}{\varepsilon}\sup_{v\in B_{1}}\left|I_1(\sigma v)-I_1((\sigma-\varepsilon)v)\right|\\
	& \leq\frac{1}{\varepsilon}\sup_{v\in B_{1}}\left|\frac{(\sigma-\varepsilon)^{q}-\sigma^{q}}{q}\int_{\Omega}\mu(x)F^{q}(\nabla v)\diff x+I_2(\sigma v)-I_2((\sigma-\varepsilon)v)\right|.
	\end{align*}
	Applying \textnormal{(H)(ii)} and the continuous embedding $\Wpzero{p}\hookrightarrow L^{p^{*}}(\Omega)$ we obtain 
	\begin{align*}
	\frac{\psi_{\lambda}(\sigma)-\psi_{\lambda}(\sigma-\varepsilon)}{\varepsilon} 
	&\leq  \frac{1}{\varepsilon}\sup_{\|v\|_{1,{p},0,F}\leq1}\int_{\Omega}\l|\int_{(\sigma-\varepsilon)v_{+}(x)}^{\sigma v_{+}(x)}\l[ t^{p^{*}-1}+\lambda t^{\gamma-1}+\lambda g(t)\r]\diff t\r|\diff x\\
	& \leq \frac{\kappa_{p^{*}}^{p^{*}}}{p^{*}}\l|\frac{\sigma^{p^{*}}-(\sigma-\varepsilon)^{p^{*}}}{\varepsilon}\r|+\lambda\frac{\kappa_{p^{*}}^{\gamma}|\Omega|^{\frac{p^{*}-\gamma}{p^{*}}}}{\gamma}\l|\frac{\sigma^{\gamma}-(\sigma-\varepsilon)^{\gamma}}{\varepsilon}\r|\\
	& \quad +\lambda c_{1}\frac{\kappa_{p^{*}}^{\nu}|\Omega|^{\frac{p^{*}-\nu}{p^{*}}}}{\nu}\l|\frac{\sigma^{\nu}-(\sigma-\varepsilon)^{\nu}}{\varepsilon}\r|
	+  \lambda c_{2}\frac{\kappa_{p^{*}}^{\theta}|\Omega|^{\frac{p^{*}-\theta}{p^{*}}}}{\theta}\l|\frac{\sigma^{\theta}-(\sigma-\varepsilon)^{\theta}}{\varepsilon}\r|.
	\end{align*}
	Therefore 
	\begin{align*}
	&\limsup_{\varepsilon\rightarrow 0^+}\frac{\psi_{\lambda}(\sigma)-\psi_{\mu}(\sigma-\varepsilon)}{\varepsilon}\\
	&\leq  \kappa_{p^{*}}^{p^{*}}\sigma^{p^{*}-1}+\lambda\kappa_{p^{*}}^{\gamma}|\Omega|^{\frac{p^{*}-\gamma}{p^{*}}}\sigma^{\gamma-1}+\lambda c_{1}\kappa_{p^{*}}^{\nu}|\Omega|^{\frac{p^{*}-\nu}{p^{*}}}\sigma^{\nu-1}+\lambda c_{2}\kappa_{p^{*}}^{\theta}|\Omega|^{\frac{p^{*}-\theta}{p^{*}}}\sigma^{\theta-1}.
	\end{align*}
	
	We now consider the function $\Lambda\colon(0,+\infty)\to \R$ defined by
	\begin{align*}
	\Lambda(s)=\frac{s^{p-\gamma}-\kappa_{p^{*}}^{p^*}s^{p^*-\gamma}}{\kappa_{p^{*}}^\gamma |\Omega|^\frac{p^*-\gamma}{p^*}+c_1 \kappa_{p^{*}}^\nu |\Omega|^\frac{p^*-\nu}{p^*}s^{\nu-\gamma}+c_2\kappa_{p^{*}}^\theta |\Omega|^\frac{p^*-\theta}{p^*}s^{\theta-\gamma}}.
	\end{align*}
	
	Applying L'Hospital's rule, it is clear that $\displaystyle \lim_{s\to 0}\Lambda(s)=0$ while $\displaystyle \lim_{s\to \infty}\Lambda(s)=-\infty$. On the other hand, there exists $s_0>0$ small enough such that $\Lambda(s)>0$ for every $s\in (0,s_0)$.  Thus, there exists $s_{\rm max}>0$ such that 
	\begin{align*}
	\Lambda(s_{\rm max})=\max_{s>0}\Lambda(s).
	\end{align*}
	Therefore we may choose 
	\begin{align*}
		\lambda_*:=\Lambda\left(\min\{s_{\rm max},\sigma^*\}\right).
	\end{align*}
	Hence, if $\lambda<\lambda_*$ and $\sigma<\min\{s_{\rm max},\sigma^*\}$, we have that \eqref{eq:ezkell} and so \eqref{min} are fulfilled.
	
	From \eqref{min} we know that we can find $u\in\W$ with $\|u\|_{1,p,0,F}\leq \sigma$ such that 
	\begin{align}\label{eq:u0}
	J_\lambda(u)<\frac{1}{p}\sigma^{p}-I_1(u_1)\quad\text{for all }u_1\in B_{\sigma}.
	\end{align}
	Since $J_{\lambda}\big|_{B\sigma}$ is sequentially weakly lower semicontinuous by Proposition \ref{proposition_3.1}, the energy functional $J_\lambda\colon \W\to\R$ restricted to $B_\sigma$ has a global minimizer $u_* \in \W$ with $\|u_*\|_{1,{p},0,F}\leq \sigma$. If $\|u_*\|_{1,p,0,F}=\sigma$, then from \eqref{eq:u0} we obtain
	\begin{align*}
	J_\lambda(u_*)=\frac{1}{p}\sigma^{p}-I_1(u_{\ast})>J_{\lambda}(u),
	\end{align*}
	which is a contradiction. It follows that $u_{\ast}$ is a local minimum for $J_{\lambda}$ with $\|u_*\|_{1,{p},0,F}<\sigma$. In summary, the energy functional $J_\lambda$ has a local minimum $u=:u_*\in B_\sigma$ for $\lambda<\lambda_*$. 
	
	Due to the presence of the singular term, zero is not a local minimizer of $J_\lambda$. Indeed, if $v\in\W$ is positive in $\Omega$ and $\tau>0$, then,
	\begin{align*}
	&J_\lambda(\tau v)\\
	& \leq \frac{\tau^p}{p}\|v\|_{1,p,0,F}^p+\frac{\tau^q}{q}\int_{\Omega}\mu(x)F^q(\nabla v)\diff x-\frac{  \tau^{p^*}}{p^*}\|v\|^{p^*}_{p^*}-\lambda\frac{ \tau^\gamma}{\gamma}\|v\|_\gamma^\gamma+ \lambda\frac{c_1\tau^\nu }{\nu}\|v\|_\nu^\nu+ \lambda\frac{c_2\tau^\theta }{\theta}\|v\|_\theta^\theta.
	\end{align*}
	Thus $J_\lambda(\tau v) <0$ for sufficiently small $\tau>0$.
	
	It remains to prove that $u$ is a weak solution of \eqref{problem}. First, we prove that $u>0$ a.\,e.\,in $\Omega$. To this end, for $t>0$ small enough,  one has $u+t u_{-}\in B_\sigma$ and $(u+t u_{-})_+=u_+$. Thus, 
	\begin{align*}
	0
	&\leq \frac{J_\lambda(u+tu_{-})- J_\lambda(u)}{t}\\ 
	&=\frac{1}{p}\left(\frac{\|u+tu_{-}\|_{1,p,0,F}^p-\|u\|_{1,p,0,F}^p}{t}\right)+\frac{1}{q}\int_\Omega
	\mu(x)\frac{F^q(\nabla(u+tu_-))-F^q(\nabla u)}{t}\diff x\\
	&\to \int_\Omega F^{p-1}(\nabla u)\nabla F( \nabla u)\cdot\nabla u_{-}\diff x+\int_\Omega \mu(x)F^{q-1}(\nabla u)\nabla F(\nabla u)\cdot\nabla u_-\diff x\quad\text{as }t \to 0^+.
	\end{align*}
	On the other hand, using Proposition \ref{basic-properties} \textnormal{(iii)}, one has that 
	\begin{align*}
	\int_\Omega F^{p-1}(\nabla u)\nabla F( \nabla u)\cdot\nabla u_{-}\diff x&=-\int_\Omega F^{p-1}(\nabla u_-)\nabla F( \nabla u_-)\cdot\nabla u_{-}\diff x=-\int_\Omega F^p(\nabla u_-)\diff x,
	\end{align*}
	and
	\begin{align*}
	\int_\Omega \mu(x)F^{q-1}(\nabla u)\nabla F(\nabla u)\cdot\nabla u_-\diff x=-\int_\Omega \mu(x)F^{q}(\nabla u_-)\diff x.
	\end{align*}
	Therefore 
	\begin{align*}
	0\leq \lim_{t\to 0} \frac{J_\lambda(u+tu_{-})- J_\lambda(u)}{t}=-\int_\Omega F^p(\nabla u_-)\diff x -\int_\Omega \mu(x)F^{q}(\nabla u_-)\diff x,
	\end{align*}
	which shows that $u_{-}=0$ and so $u\geq 0$ a.\,e.\,in $\Omega$.
	
	Assume that there exists a set $A$ of positive measure such that $u=0$ in $A$.  Let $\varphi\colon\Omega\to \R$ be a function in $\W$ which is positive in $\Omega$. For $t>0$ small enough, we know that $u+t\varphi \in B_\sigma$ and $(u+t\varphi)^\gamma>u^\gamma$ a.\,e.\,in $\Omega$. Hence, we obtain
	\begin{align*}
	\begin{split}
	0& \leq \frac{J_\lambda (u+t\varphi)- J_\lambda(u)}{t}\\
	& = \frac{1}{p}\left(\frac{\|u+t\varphi\|^p_{1,p,0,F}-\|u\|^p_{1,p,0,F}}{t}\right)+\frac{1}{q}\int_\Omega	\mu(x)\frac{F^q(\nabla(u+t\varphi))^{q}-F^q(\nabla u)}{t}\diff x\\ 
	&\quad -\frac{1}{p^*}\int_\Omega \frac{(u+t\varphi)^{p^*}-u^{p^*}}{t}\diff x-\frac{\lambda}{\gamma t^{1-\gamma}}\int_A \varphi^\gamma\diff x\\
	&\quad -\frac{\lambda}{\gamma}\int_{\Omega\setminus A} \frac{(u+t\varphi)^{\gamma}-u^{\gamma}}{t}\diff x- \lambda\int_\Omega \frac{G((u+t\varphi))-G(u)}{t}\diff x \\
	&< \frac{1}{p}\left(\frac{\|u+t\varphi\|^p_{1,p,0,F}-\|u\|^p_{1,p,0,F}}{t}\right)+\frac{1}{q}\int_\Omega	\mu(x)\frac{F^q(\nabla(u+t\varphi))^{q}-F^q(\nabla u)}{t}\diff x\\
	& \quad -\frac{1}{p^*}\int_\Omega \frac{(u+t\varphi)^{p^*}-u^{p^*}}{t}\diff x -\frac{\lambda}{\gamma t^{1-\gamma}}\int_A \varphi^\gamma\diff x- \lambda\int_\Omega \frac{G((u+t\varphi))-G(u)}{t}\diff x.
	\end{split}
	\end{align*}
	Hence,
	\begin{align}\label{ezmajdlentkell}
	\begin{split}
	0& \leq \frac{J_\lambda (u+t\varphi)- J_\lambda(u)}{t}
	\longrightarrow -\infty \quad \text{as} \quad  t\to 0^+,
	\end{split}
	\end{align}
	which is a contradiction. Thus, $u>0$ a.\,e.\,in $\Omega$.
	
	In the next step we are going to prove that
	\begin{equation}\label{def1}
	u^{\gamma-1}\varphi\in L^1(\Omega) \quad \text{for all } \varphi\in\W
	\end{equation}
	and
	\begin{align}\label{def2}
	\begin{split}
	& \int_{\Omega}\left(F(\nabla u)^{p-1}+\mu(x)F(\nabla u)^{q-1}\right)\nabla F(\nabla u)\cdot\nabla\varphi\diff x\\
	& -\int_{\Omega}u^{p^{*}-1}\varphi\diff x-\lambda\int_{\Omega}u^{\gamma-1}\varphi\diff x-\lambda\int_{\Omega}g(u)\varphi\diff x\geq 0\quad \text{for all } \ph \in \W \text{ with }\ph \geq 0.
	\end{split}
	\end{align}
	Choose $\varphi\in \W$ with $\varphi\geq 0$ and fix a decreasing sequence $\{t_n\}_{n \in\N} \subseteq (0,1]$ such that $\displaystyle \lim_{n\to \infty} t_n=0$.  For $n \in\N$, the functions
	\begin{align*}
	h_n(x)=\frac{(u(x)+t_n\varphi(x))^\gamma-u(x)^\gamma}{t_n}
	\end{align*}
	are measurable, nonnegative  and we have
	\begin{align*}
	\lim_{n\to \infty} h_n(x)=\gamma u(x)^{\gamma-1}\varphi(x)\quad \text{for a.\,a.\,} x\in\Omega.
	\end{align*}
	From  Fatou's lemma we infer
	\begin{equation}\label{fatou}
	\int_\Omega u^{\gamma-1}\varphi\diff x\leq \frac{1}{\gamma}\liminf_{n\to\infty}\int_\Omega h_n\diff x.
	\end{equation}
	Similar to \eqref{ezmajdlentkell}, we have for $n\in\N$ large enough
	\begin{align*}
	0& \leq \frac{J_\lambda(u+t_n\varphi)-J_\lambda(u) }{t_n}\\
	&=\frac{1}{p}\frac{\|u+t_n\varphi\|_{1,p,0,F}^p-\|u\|_{1,p,0,F}^p}{t_n}+\frac{1}{q}\int_\Omega	\mu(x)\frac{F^q(\nabla(u+t\varphi))^{q}-F^q(\nabla u)}{t}\diff x\\
	&\quad -\frac{1}{p^*}\int_\Omega
	\frac{(u+t_n\varphi)^{p^*}-u^{p^*}}{t_n}\diff x\ -\frac{\lambda}{\gamma}\int_\Omega h_n\diff x-\lambda\int_\Omega\frac{G(u+t_n\varphi)-G(u)}{t_n}\diff x.
	\end{align*}
	Passing to the limit as $n\to \infty$ and applying \eqref{fatou} we have \eqref{def1} and we get that
	\begin{align*}
	\lambda\int_\Omega u^{\gamma-1}\varphi\diff x
	& \leq\int_{\Omega}\left(F(\nabla u)^{p-1}+\mu(x)F(\nabla u)^{q-1}\right)\nabla F(\nabla u)\cdot \nabla\varphi\diff x\\ 
	& \quad -\int_{\Omega}u^{p^{*}-1}\varphi\diff x-\lambda\int_{\Omega}g(u)\varphi\diff x,
	\end{align*}
	which shows \eqref{def2}. Note that is it enough to show \eqref{def1} for nonnegative $\ph\in \W$.
	
	Now, let $\varepsilon\in (0,1)$ be such that $(1+t)u\in B_\sigma$ for all $t\in[-\varepsilon, \varepsilon]$. Note that the function $\alpha(t):=J_\lambda((1+t)u)$ has a local minimum in zero. Using again Proposition \ref{basic-properties} \textnormal{(iii)} gives
	\begin{align}\label{alpha}
	\begin{split}
	0&=\alpha'(0)=\lim_{t\to 0}\frac{J_\lambda((1+t)u)-J_\lambda(u)}{t}\\
	& = \rho_{\mathcal{H},F}(u)-\int_\Omega u^{p^*}\diff x-\lambda\int_\Omega u^{\gamma}\diff x-\lambda\int_\Omega g(u)u\diff x.
	\end{split}
	\end{align}
	
	It remains to show that  $u$ is indeed a positive weak solution of \eqref{problem}. To this end we use the method of Lair-Shaker, see, for example, Lair-Shaker \cite{Lair-Shaker-1997} or Sun-Wu-Long \cite{Sun-Wu-Long-2001}. Let $\varphi\in\W$ and take the test function $v=(u+\varepsilon\varphi)_+$ in \eqref{def2}. Applying \eqref{alpha}  we have that
	\begin{align}\label{estimate}
	\begin{split}
	0
	&\leq \int_{\{u+\varepsilon\varphi\geq 0\}}\left(F^{p-1}(\nabla u)+\mu(x)F^{q-1}(\nabla u)\right)\nabla F(\nabla u)\cdot \nabla(u+\varepsilon\varphi)\diff x\\
	& \quad -\int_{\Omega}u^{p^{*}-1}(u+\varepsilon\varphi)\diff x -\lambda\int_{\Omega}u^{\gamma-1}(u+\varepsilon\varphi)\diff x-\lambda\int_{\Omega}g(u)(u+\varepsilon\varphi)\diff x\\
	&= \rho_{\mathcal{H},F}(u)+\varepsilon \int_\Omega F^{p-1}(\nabla u)\nabla F(\nabla u)\cdot\nabla \varphi \diff x+\varepsilon \int_\Omega\mu(x) F^{q-1}(\nabla u)\nabla F(\nabla u)\cdot\nabla \varphi \diff x\\
	& \quad -\int_\Omega u^{p^*}\diff x - \varepsilon  \int_\Omega u^{p^*-1}\varphi\diff x - \lambda\int_\Omega u^\gamma \diff x-\varepsilon \lambda\int_\Omega u^{\gamma-1}\diff x\\
	&\quad -\lambda\int_\Omega g(u)u\diff x-\varepsilon\lambda\int_\Omega g(u)\varphi\diff x -\int_{\{u+\varepsilon\varphi< 0\}}F^p(\nabla u)\diff x\\ &\quad-\varepsilon\int_{\{u+\varepsilon\varphi< 0\}}F^{p-1}(\nabla u)\nabla F(\nabla u)\cdot\nabla \varphi \diff x
	-\int_{\{u+\varepsilon\varphi< 0\}}\mu(x)F^q(\nabla u)\diff x\\ & \quad -\varepsilon\int_{\{u+\varepsilon\varphi< 0\}}\mu(x) F^{q-1}(\nabla u)\nabla F(\nabla u)\cdot\nabla \varphi \diff x +\int_{\{u+\varepsilon\varphi< 0\}} u^{p^*-1}(u+\varepsilon\varphi)\diff x\\
	&\quad +\lambda\int_{\{u+\varepsilon\varphi<0\}} u^{\gamma-1}(u+\varepsilon\varphi)\diff x+\lambda\int_{\{u+\varepsilon\varphi< 0\}} g(u)(u+\varepsilon\varphi)\diff x\\ 
	&\leq \varepsilon\left[ \int_\Omega \l(F^{p-1}(\nabla u)+\mu(x) F^{q-1}(\nabla u)\r)\nabla F(\nabla u)\cdot\nabla \varphi \diff x-\int_\Omega u^{p^*-1}\varphi\diff x\r.\\
	& \l. \qquad-\lambda\int_\Omega u^{\gamma-1}\varphi\diff x-\lambda\int_\Omega g(u) \varphi\diff  x\right]-\varepsilon\int_{\{u+\varepsilon\varphi< 0\}}F^{p-1}(\nabla u)\nabla F(\nabla u)\cdot \nabla \varphi\diff x\\
	& \quad -\varepsilon\int_{\{u+\varepsilon\varphi< 0\}}\mu(x) F^{q-1}(\nabla u)\nabla F(\nabla u)\cdot\nabla \varphi \diff x.
	\end{split}
	\end{align}
	Note that the measure of the set $\{u+\varepsilon\varphi< 0\}$ goes to $0$ as $\varepsilon\to 0$. Hence,
	\begin{align*}
	\int_{\{u+\varepsilon\varphi< 0\}}F^{p-1}(\nabla u)\nabla F(\nabla u)\nabla \varphi\diff x+\int_{\{u+\varepsilon\varphi< 0\}}\mu(x) F^{q-1}(\nabla u)\nabla F(\nabla u)\nabla \varphi \diff x\to 0.
	\end{align*}
	Using this, dividing \eqref{estimate} by $\varepsilon$ and passing to the limit as $\varepsilon\to 0$ we conclude that 
	\begin{align*}
	& \int_{\Omega}\left(F(\nabla u)^{p-1}+\mu(x)F(\nabla u)^{q-1}\right)\nabla F(\nabla u)\cdot \nabla\varphi\diff x\\  
	& -\int_{\Omega}u^{p^{*}-1}\varphi\diff x-\lambda\int_{\Omega}u^{\gamma-1}\varphi\diff x-\lambda\int_{\Omega}g(u)\varphi\diff x\geq 0.
	\end{align*}
	As $\ph$ was arbitrary chosen, we get from the last inequality that $u \in \W$ is a weak solution of \eqref{problem}. This finishes the proof.
\end{proof}

\section*{Acknowledgment}
C.\,Farkas was supported by the National Research, Development and Innovation Fund of Hungary, financed under the K\_18 funding scheme, Project No.\,127926 and by the Sapientia Foundation - Institute for Scientific Research, Romania, Project No.\,17/11.06.2019.


\begin{thebibliography}{99}
	
	\bibitem{Bahrouni-Radulescu-Winkert-2020}
	A.~Bahrouni, V.D.~R\u{a}dulescu, P. Winkert,
	{\it Double phase problems with variable growth and convection for the Baouendi-Grushin operator}, 
	Z. Angew. Math. Phys. {\bf 71} (2020), no. 6, 183.
	
	\bibitem{Bao-Chern-Shen-2000}
	D.~Bao, S.-S. Chern, Z.~Shen,  
	``An Introduction to Riemann-Finsler Geometry'', 
	Springer-Verlag, New York, 2000.
	
	\bibitem{Baroni-Colombo-Mingione-2015}
	P.~Baroni, M.~Colombo, G.~Mingione,
	{\it Harnack inequalities for double phase functionals},
	Nonlinear Anal. {\bf 121} (2015), 206--222.
	
	\bibitem{Baroni-Colombo-Mingione-2016}
	P.~Baroni, M.~Colombo, G.~Mingione,
	{\it Non-autonomous functionals, borderline cases and related function classes},
	St. Petersburg Math. J. {\bf 27} (2016), 347--379.
	
	\bibitem{Baroni-Colombo-Mingione-2018}
	P.~Baroni, M.~Colombo, G.~Mingione,
	{\it Regularity for general functionals with double phase},
	Calc. Var. Partial Differential Equations {\bf 57} (2018), no. 2, Art. 62, 48 pp.
	
	\bibitem{Baroni-Kuusi-Mingione-2015}
	P.~Baroni, T.~Kuusi, G.~Mingione,
	{\it Borderline gradient continuity of minima},
	J. Fixed Point Theory Appl. {\bf 15} (2014), no. 2, 537--575.
	
	\bibitem{Bellettini-Paolini-1996}
	G.~Bellettini, M.~Paolini,
	{\it Anisotropic motion by mean curvature in the context of {F}insler geometry},
	Hokkaido Math. J.  {\bf 25 } (1996),  no. 3, 537-566.
	
	\bibitem{Belloni-Ferone-Kawohl-2003}
	M.~Belloni, V.~Ferone, B.~Kawohl, 
	{\it Isoperimetric inequalities, {W}ulff shape and related questions for strongly nonlinear elliptic operators}, 
	Z. Angew. Math. Phys.  {\bf 54} (2003),  no. 5, 771-783.
	
	\bibitem{Chen-Ge-Wen-Cao-2020}
	Z.-Y.~Chen, B.~Ge, W.-S. Yuan, X.-F. Cao,
	{\it Existence of solution for double-phase problem with singular weights},
	Adv. Math. Phys. {\bf 2020} (2020), Art. ID 5376013, 7 pp.
	
	\bibitem{Cianchi-Salani} 
	A.~Cianchi, P.~Salani,
	{\it Overdetermined anisotropic elliptic problems},
	Math. Ann. {\bf 345} (2009), no. 4, 859--881.
	
	\bibitem{Colasuonno-Squassina-2016}
	F.~Colasuonno, M.~Squassina,
	{\it Eigenvalues for double phase variational integrals},
	Ann. Mat. Pura Appl. (4) {\bf 195} (2016), no. 6, 1917--1959.
	
	\bibitem{Colombo-Mingione-2015a}
	M.~Colombo, G.~Mingione,
	{Bounded minimisers of double phase variational integrals},
	Arch. Ration. Mech. Anal. {\bf 218} (2015), no. 1, 219--273.
	
	\bibitem{Colombo-Mingione-2015b}
	M.~Colombo, G.~Mingione,
	{\it Regularity for double phase variational problems},
	Arch. Ration. Mech. Anal. {\bf 215} (2015), no. 2, 443--496.
	
	\bibitem{DellaPietra-diBlasio-Gavitone-2020}
	F.~Della Pietra, G.~di Blasio, N.~Gavitone,
	{\it Sharp estimates on the first {D}irichlet eigenvalue of nonlinear elliptic operators via maximum principle}, 
	Adv. Nonlinear Anal. {\bf 9} (2020), no. 1, 278--291.
	
	\bibitem{DellaPietra-Gavitone-Piscitelli-2019}
	F.~Della Pietra, N.~Gavitone, G.~Piscitelli,
	{\it On the second {D}irichlet eigenvalue of some nonlinear anisotropic elliptic operators}
	Bull. Sci. Math. {\bf 155} (2019), 10--32.
	
	\bibitem{DellaPietra-Gavitone-2016}
	F.~Della Pietra, N.~Gavitone,
	{\it Sharp estimates and existence for anisotropic elliptic problems with general growth in the gradient}, 
	Z. Anal. Anwend. {\bf 35} (2016), no. 1, 61--80.
	
	\bibitem{Dinghas-1944}
	A.~Dinghas,  
	{\it \"{U}ber einen geometrischen {S}atz von {W}ulff f\"{u}r die {G}leichgewichtsform von {K}ristallen} ,
	Z. Kristallogr., Mineral. Petrogr. {\bf 105} (1944), Abt. A., 304--314.
	
	\bibitem{Fan-2012}
	X.~Fan,
	{\it An imbedding theorem for Musielak-Sobolev spaces},
	Nonlinear Anal. {\bf 75} (2012), no. 4, 1959--1971.
	
	\bibitem{Faraci-Farkas-2015}
	F.~Faraci, C.~Farkas,
	{\it A quasilinear elliptic problem involving critical {S}obolev exponents},
	Collect. Math. {\bf 66} (2015), no. 2, 243--259.
	
	\bibitem{Farkas-2020}
	C.~Farkas, 
	{\it Critical elliptic equations  on non-compact Finsler manifolds}, 
	arXiv:2010.07686, 2020.
	
	\bibitem{Farkas-Fodor-Kristaly-2015} 
	C.~Farkas, J.~Fodor and A.~Krist\'aly, 
	{\it Anisotropic elliptic problems involving sublinear terms}, 
	2015 IEEE 10th Jubilee International Symposium on Applied Computational Intelligence and Informatics, Timisoara, 2015, pp. 141--146.
	
	\bibitem{Farkas-Kristaly-Varga-2015}
	C.~Farkas, A.~Krist\'{a}ly, C.~Varga,
	{\it Singular {P}oisson equations on {F}insler-{H}adamard manifolds},
	Calc. Var. Partial Differential Equations {\bf 54} (2015), no. 2, 1219--1241.
	
	\bibitem{Ferone-Kawohl-2009}
	V.~Ferone,  B.~Kawohl, 
	{\it Remarks on a {F}insler-{L}aplacian},
	Proc. Amer. Math. Soc.  {\bf 137}  (2009),  no. 1, 247-253.
	
	\bibitem{Gasinski-Papageorgiou-2019}
	L.~Gasi\'nski, N.S.~Papageorgiou,
	{\it Constant sign and nodal solutions for superlinear double phase problems},
	Adv. Calc. Var., https://doi.org/10.1515/acv-2019-0040.
	
	\bibitem{Gasinski-Winkert-2020a}
	L.~Gasi\'{n}ski, P.~Winkert,
	{\it Constant sign solutions for double phase problems with superlinear nonlinearity},
	Nonlinear Anal. {\bf 195} (2020), 111739.
	
	\bibitem{Gasinski-Winkert-2020b}
	L.~Gasi\'nski, P.~Winkert,
	{\it Existence and uniqueness results for double phase problems with convection term},
	J. Differential Equations {\bf 268} (2020), no. 8, 4183--4193.
	
	\bibitem{Gasinski-Winkert-2021}
	L.~Gasi\'nski, P.~Winkert,
	{\it Sign changing solution for a double phase problem with nonlinear boundary condition via the Nehari manifold},
	J. Differential Equations {\bf 274} (2021), 1037--1066.
	
	\bibitem{Lair-Shaker-1997}
	A.V.~Lair, A.W.~Shaker, 
	{\it Classical and weak solutions of a singular semilinear elliptic problem}, 
	J. Math. Anal. Appl. {\bf 211}(1997), no. 2, 193--222.
	
	\bibitem{Lindqvist-1990}
	P.~Lindqvist,
	{\it On the equation {${\rm div}\,(|\nabla u|^{p-2}\nabla u)+\lambda|u|^{p-2}u=0$}},
	Proc. Amer. Math. Soc. {\bf 109} (1990), no. 1, 157--164.
	
	\bibitem{Liu-Dai-2018}
	W.~Liu, G.~Dai,
	{\it Existence and multiplicity results for double phase problem},
	J. Differential Equations {\bf 265} (2018), no. 9, 4311--4334.
	
	
	
	\bibitem{Marcellini-1991}
	P.~Marcellini,
	{\it Regularity and existence of solutions of elliptic equations with {$p,q$}-growth conditions},
	J. Differential Equations {\bf 90} (1991), no. 1, 1--30.
	
	\bibitem{Marcellini-1989}
	P.~Marcellini,
	{\it The stored-energy for some discontinuous deformations in nonlinear elasticity}, in ``Partial differential equations and the calculus of variations, {V}ol. {II}'', vol. 2, 767--786, Birkh\"{a}user Boston, Boston, 1989.
	
	\bibitem{Marino-Winkert-2020}
	G.~Marino, P.~Winkert,
	{\it Existence and uniqueness of elliptic systems with double phase operators and convection terms},
	J. Math. Anal. Appl. {\bf 492} (2020), 124423, 13 pp.
	
	\bibitem{Ohta-2009}
	S.-I.~Ohta,
	{\it Uniform convexity and smoothness, and their applications in {F}insler geometry},
	Math. Ann. {\bf 343} (2009), no. 3, 669--699.
	
	\bibitem{Ok-2018}
	J. Ok,
	{\it Partial regularity for general systems of double phase type with continuous coefficients},
	Nonlinear Anal. {\bf 177} (2018), 673--698.
	
	\bibitem{Ok-2020}
	J.~Ok,
	{\it Regularity for double phase problems under additional integrability assumptions},
	Nonlinear Anal. {\bf 194} (2020), 111408.
	
	\bibitem{Papageorgiou-Radulescu-Repovs-2020}
	N.S.~Papageorgiou, V.D.~R\u{a}dulescu, D.D.~Repov\v{s},
	{\it Ground state and nodal solutions for a class of double phase problems},
	Z. Angew. Math. Phys. {\bf 71} (2020), no. 1, 15 pp.
	
	\bibitem{Papageorgiou-Winkert-2018}
	N.S.~Papageorgiou, P.~Winkert,
	``Applied Nonlinear Functional Analysis. An Introduction'',
	De Gruyter, Berlin, 2018.
	
	\bibitem{Papageorgiou-Winkert-2021}
	N.S.~Papageorgiou, P.~Winkert,
	{\it Singular Dirichlet $(p,q)$-equations},
	Mediterr. J. Math., accepted for publication.
	
	\bibitem{Papageorgiou-Winkert-2019}
	N.S.~Papageorgiou, P.~Winkert,
	{\it Singular {$p$}-{L}aplacian equations with superlinear perturbation},
	J. Differential Equations {\bf 266} (2019), no. 2-3, 1462--1487.
	
	\bibitem{Perera-Squassina-2019}
	K.~Perera, M.~Squassina,
	{\it Existence results for double-phase problems via Morse theory},
	Commun. Contemp. Math. {\bf 20} (2018), no. 2, 1750023, 14 pp.
	
	\bibitem{Ragusa-Tachikawa-2020}
	M.A.~Ragusa, A.~Tachikawa,
	{\it Regularity for minimizers for functionals of double phase with variable exponents},
	Adv. Nonlinear Anal. {\bf 9} (2020), no. 1, 710--728.
	
	\bibitem{Randers-1941} 
	G.~Randers, 
	{\it On an asymmetrical metric in the fourspace of general relativity}
	Phys. Rev. (2)  {\bf 59} (1941), 195--199.
	
	\bibitem{Sun-Wu-Long-2001}
	Y.~Sun, S.~Wu, Y.~Long, 
	{\it Combined effects of singular and superlinear nonlinearities in some singular boundary value problems}, 
	J. Differential Equations {\bf 176} (2001), no. 2, 511--531.
	
	\bibitem{Taylor-2002} 
	J.E.~Taylor, 
	{\it Crystalline variational methods},
	Proc. Natl. Acad. Sci. USA {\bf 99} (2002),  no. 24, 15277-15280.
	
	\bibitem{Wang-Xia-2011}
	G.~Wang, C.~Xia,  
	{\it A characterization of the {W}ulff shape by an overdetermined anisotropic {PDE}},
	Arch. Ration. Mech. Anal. {\bf 199} (2011), no. 1, 99-115.
	
	\bibitem{Zeng-Bai-Gasinski-Winkert-2020}
	S.D.~Zeng, Y.R.~Bai, L.~Gasi\'{n}ski, P.~Winkert,
	{\it Existence results for double phase implicit obstacle problems involving multivalued operators},
	Calc. Var. Partial Differential Equations {\bf 59} (2020), no. 5, 176.
	
	\bibitem{Zeng-Gasinski-Winkert-Bai-2020}
	S.D.~Zeng, L.~Gasi\'{n}ski, P.~Winkert, Y.R.~Bai,
	{\it Existence of solutions for double phase obstacle problems with multivalued convection term},
	J. Math. Anal. Appl., https://doi.org/10.1016/j.jmaa.2020.123997.
	
	\bibitem{Zhikov-1986}
	V.~V.~Zhikov,
	{\it Averaging of functionals of the calculus of variations and elasticity theory},
	Izv. Akad. Nauk SSSR Ser. Mat. {\bf 50} (1986), no. 4, 675--710.
	
	\bibitem{Zhikov-Kozlov-Oleinik-1994}
	V.~V.~Zhikov, S.~M.~Kozlov, O.~A.~Ole\u{\i}nik,
	``Homogenization of Differential Operators and Integral Functionals'',
	Springer-Verlag, Berlin, 1994.
	
\end{thebibliography}
\end{document}